\documentclass[11pt, reqno]{amsart}

\usepackage[utf8]{inputenc}
\usepackage[margin=0.5in]{geometry}
\usepackage{lipsum, fullpage,comment, textcomp, gensymb, amssymb, multirow, graphicx, amsmath, mathtools, soul, amsfonts, hyperref, bm, caption, mathrsfs, multicol, multirow, float, pgf, tikz, pgfplots, wrapfig, amsthm, enumerate, enumitem, dsfont, array, booktabs, framed, tikz-cd, xfrac, tabularx, extarrows, accents,  old-arrows, xparse, csquotes, microtype, thmtools, nicematrix}
\pgfplotsset{compat=1.18}
\usepackage[toc,page]{appendix} 

\usepackage[backend=biber, style=alphabetic, maxnames=99, minnames=99, sorting=nyt]{biblatex}
\usepackage[UKenglish]{babel}
\usepackage[table]{xcolor}

\usetikzlibrary {arrows.meta,bending,positioning,patterns,decorations}
\usepackage{pbox}
\usepackage{tkz-graph}
\usetikzlibrary{decorations.pathmorphing,fit,positioning,calc,shapes}
\usepackage{graphicx}
\usepackage{cleveref,tikz-cd,pbox,wrapfig}
\definecolor{amethyst}{rgb}{0.6, 0.4, 0.8}
\definecolor{atomictangerine}{rgb}{1.0, 0.6, 0.4}
\definecolor{deeppeach}{rgb}{1.0, 0.8, 0.64}
\definecolor{eggshell}{rgb}{0.94, 0.92, 0.84}
\definecolor{lightapricot}{rgb}{0.99, 0.84, 0.69}
\definecolor{lemonchiffon}{rgb}{1.0, 0.98, 0.8}
\definecolor{roundabout}{rgb}{1.0, 0.91, 0.75}
\definecolor{atomictangerine}{rgb}{1.0, 0.6, 0.4}
\definecolor{ruby}{rgb}{0.88, 0.07, 0.37}
\definecolor{sapphire}{rgb}{0.03, 0.15, 0.4}
\definecolor{LightBlue}{rgb}{0.0, 0.48, 0.65}

\def\rootsep{0.03}               
\def\clustersep{0.06}            
\def\cnamescale{0.4}             
\def\cdepthscale{0.4}            
\def\cltopskip{1pt}              
\def\clbottomskip{1pt}           

\def\rootscale{0.5}   \def\rootcolor{gray}
\def\rootscaleA{0.7}  \def\rootcolorA{yellow}
\def\rootscaleB{0.5}  \def\rootcolorB{green}
\def\rootscaleC{0.4}  \def\rootcolorC{sapphire}
\def\rootscaleD{0.45}  \def\rootcolorD{blue}
\tikzset{
  clA/.style = {very thick,black},
  clB/.style = {thick,purple}
}

\def\graphdslabelscale{0.6}

\def\GraphScale{0.6}

\tikzset{
  root/.style = {circle,scale=\rootscale,fill=\rootcolor},
    rc/.style 2 args = {right=#1*1.5*\clustersep of {#2.east|-first},root}, rr/.style = {right=\rootsep of {#1.east|-first},root},
  roott/.style = {circle,inner sep=-2pt,minimum size=5pt,black,font=\ttfamily\footnotesize},
    rct/.style 2 args = {right=#1*1.5*\clustersep of {#2.east|-first},roott}, rrt/.style = {right=\rootsep of {#1.east|-first},roott},
  rootA/.style = {circle,scale=\rootscaleA,ball color=\rootcolorA},
    rcA/.style 2 args = {right=#1*1.5*\clustersep of {#2.east|-first},rootA}, rrA/.style = {right=\rootsep of {#1.east|-first},rootA},
  rootB/.style = {circle,scale=\rootscaleB,ball color=\rootcolorB},
    rcB/.style 2 args = {right=#1*1.5*\clustersep of {#2.east|-first},rootB}, rrB/.style = {right=\rootsep of {#1.east|-first},rootB},
  rootC/.style = {diamond,scale=\rootscaleC,ball color=\rootcolorC},
    rcC/.style 2 args = {right=#1*1.5*\clustersep of {#2.east|-first},rootC}, rrC/.style = {right=\rootsep of {#1.east|-first},rootC},
  rootD/.style = {circle,scale=\rootscaleD,ball color=\rootcolorD},
    rcD/.style 2 args = {right=#1*1.5*\clustersep of {#2.east|-first},rootD}, rrD/.style = {right=\rootsep of {#1.east|-first},rootD},
  cluster/.style = {draw=black!90,thick,rounded corners,inner sep=22*\clustersep,outer xsep=22*\clustersep,fit=#1},
  clabel/.style  = {anchor=west,scale=\cdepthscale,black,inner sep=0,outer xsep=1,outer ysep=0},
  clabelL/.style = {above right=-\clustersep of #1t.north east,clabel},
  clabelD/.style = {below right=-\clustersep of #1t.south east,clabel},
  clouter/.style = {inner sep=0,outer sep=0,fit=#1}
}

\def\Cluster #1 = #2;{\node[cluster=#2] (#1) {};}
\def\ClusterL #1[#2] = #3;{
  \node[cluster=#3] (#1t) {}; \node[clabelL=#1] (#1l) {$#2$}; \node[clouter=(#1t)(#1l)] (#1) {};}
\def\ClusterD #1[#2] = #3;{
  \node[cluster=#3] (#1t) {}; \node[clabelD=#1] (#1d) {$#2$}; \node[clouter=(#1t)(#1d)] (#1) {};}
\def\ClusterLD #1[#2][#3] = #4;{
  \node[cluster=#4] (#1t) {}; \node[clabelL=#1] (#1l) {$#2$}; 
  \node[clabelD=#1] (#1d) {$#3$}; \node[clouter=(#1t)(#1l)(#1d)] (#1) {};}
\def\ClusterLDName #1[#2][#3][#4] = #5;{
  \node[cluster=#5] (#1t) {}; \node[clabelL=#1] (#1l) {$#2$}; 
  \node[clabelD=#1] (#1d) {$#3$}; 
  \node[scale=\cnamescale,above=\clustersep/3 of #1t,inner sep=0, outer sep=0] (#1n) {$#4$}; 
  \node[clouter=(#1l)(#1d)(#1t)] (#1) {};}

\newcommand{\Root}[4][]{
  \ifx\relax#2\relax\node[rr#1=#3] (#4) {};\else\node[rc#1={#2}{#3}] (#4) {};\fi}
\newcommand{\RootT}[5][]{
  \ifx\relax#2\relax\node[rrt#1=#3] (#4) {#5};\else\node[rct#1={#2}{#3}] (#4) {#5};\fi}

\def\frob(#1)(#2){\path[draw,thick,shorten <=-22*\clustersep,shorten >=-22*\clustersep](#1.east)--(#2.west|-#1){};}

\usepackage{pbox}
\def\pb#1{\pbox[c]{\textwidth}{\hfil #1\hfil}}

\long\def\clusterpicture#1\endclusterpicture{\pb{\vbox to \cltopskip{\vfill}\\%
  \begin{tikzpicture}\node[coordinate] (first) {};#1\end{tikzpicture}\\[-11pt]\vbox to \clbottomskip{\vfill}}}   
\long\def\clusterpictureopt#1#2\endclusterpicture{\pb{\vbox to \cltopskip{\vfill}\\%
  \begin{tikzpicture}[#1]\node[coordinate] (first) {};#2\end{tikzpicture}\\[-11pt]\vbox to \clbottomskip{\vfill}}}

\def\pb#1{\pbox[c]{\textwidth}{\hfil #1\hfil}}

\def\GraphVertices{\SetVertexNormal[Shape=circle, FillColor=blue!50, LineColor=blue!50, LineWidth=0.8pt]
  \tikzset{VertexStyle/.append style = {inner sep=0.5pt,minimum size=0.3em,font = \tiny\bfseries}}}

\def\BlueEdges{  \SetUpEdge[lw=0.8pt,color=blue!70]
   \tikzset{EdgeStyle/.append style = {shorten <=0.5pt,shorten >=0.5pt}}}

\def\LoopW(#1){
  \path[draw,-,thick,color=blue!70] (#1) edge[out=155,in=90] ($(#1)-(1.3,0)$);
  \path[draw,-,thick,color=blue!70] (#1) edge[out=210,in=270] ($(#1)-(1.3,0)$);
}
\def\LoopE(#1){
  \path[draw,-,thick,color=blue!70] (#1) edge[out=25,in=90] ($(#1)+(1.3,0)$);
  \path[draw,-,thick,color=blue!70] (#1) edge[out=-25,in=270] ($(#1)+(1.3,0)$);
}
\def\LoopS(#1){
  \path[draw,-,thick,color=blue!70] (#1) edge[out=115,in=180] ($(#1)+(0,1.2)$);
  \path[draw,-,thick,color=blue!70] (#1) edge[out=65,in=0] ($(#1)+(0,1.2)$);
}
\def\LoopN(#1){
  \path[draw,-,thick,color=blue!70] (#1) edge[out=-115,in=180] ($(#1)-(0,1.2)$);
  \path[draw,-,thick,color=blue!70] (#1) edge[out=-65,in=0] ($(#1)-(0,1.2)$);
}
\def\EdgeW(#1){
  \path[draw,-,thick,color=blue!70] (#1+) edge[out=180,in=90] ($(#1+)-(1.3,0.3)$);
  \path[draw,-,thick,color=blue!70] (#1-) edge[out=180,in=270] ($(#1+)-(1.3,0.3)$);
}
\def\EdgeE(#1){
  \path[draw,-,thick,color=blue!70] (#1+) edge[out=0,in=90] ($(#1-)+(1.3,0.3)$);
  \path[draw,-,thick,color=blue!70] (#1-) edge[out=0,in=270] ($(#1-)+(1.3,0.3)$);
}
\def\EdgeS(#1){
  \path[draw,-,thick,color=blue!70] (#1+) edge[out=90,in=0] ($(#1-)+(0.3,1.3)$);
  \path[draw,-,thick,color=blue!70] (#1-) edge[out=90,in=180] ($(#1-)+(0.3,1.3)$);
}
\def\EdgeN(#1){
  \path[draw,-,thick,color=blue!70] (#1+) edge[out=270,in=0] ($(#1+)-(0.3,1.3)$);
  \path[draw,-,thick,color=blue!70] (#1-) edge[out=270,in=180] ($(#1+)-(0.3,1.3)$);
}
\def\GCircle(#1,#2)(#3,#4){
  \path(#1,#2) node[coordinate] (1) {};
  \path(#3,#4) node[coordinate] (2) {};
  \path[draw,-,thick,color=blue!70] (1) edge[out=90,in=90] (2);
  \path[draw,-,thick,color=blue!70] (2) edge[out=270,in=270] (1);
}

\def\EdgeSign(#1)(#2)#3(#4)#5{
  \node at ($(#1)!#3!(#2) + (#4)$) [color=black, scale=\graphdslabelscale] {$\scriptstyle #5$};
}

\def\GraphEdgeSignDist{0.55}

\def\GraphEdgeSignS(#1)(#2)#3#4{\EdgeSign(#1)(#2)#3(0,-\GraphEdgeSignDist){#4}}

\def\VSwap#1#2#3#4{\path[draw](#1) edge[<->,#3,shorten >=#4pt,shorten <=#4pt] (#2){};}
\def\VArr#1#2#3#4{\path[draw](#1) edge[->,#3,shorten >=#4pt,shorten <=#4pt] (#2){};}

\def\ESwapOfs#1#2#3#4#5#6#7#8{\VSwap{$(#1)!0.5!(#2) + (#6)$}{$(#3)!0.5!(#4) + (#7)$}{#5}{#8}}

\def\EArrOfs#1#2#3#4#5#6#7#8{\VArr{$(#1)!0.5!(#2) + (#6)$}{$(#3)!0.5!(#4) + (#7)$}{#5}{#8}}

\def\tgrGB{\raise-7pt\hbox{\begin{tikzpicture}[scale=\GraphScale]
  \GraphVertices
  \Vertex[x=1.50,y=0.000,L=1]{1};
  \coordinate (2) at (0.000,0.000);
  \BlueEdges
  \LoopW(1)
\GraphEdgeSignS(1)(2){0.5}{n}\end{tikzpicture}}}

\def\tgrGBex{\raise-7pt\hbox{\begin{tikzpicture}[scale=\GraphScale]
  \GraphVertices
  \Vertex[x=1.50,y=0.000,L=1]{1};
  \coordinate (2) at (0.000,0.000);
  \BlueEdges
  \LoopW(1)
\GraphEdgeSignS(1)(2){0.5}{1}\end{tikzpicture}}}

\def\tgrGC{\raise-7pt\hbox{\begin{tikzpicture}[scale=\GraphScale]
  \GraphVertices
  \Vertex[x=1.50,y=0.000,L=1]{1};
  \coordinate (2) at (0.000,0.000);
  \BlueEdges
  \LoopW(1)
\GraphEdgeSignS(1)(2){0.5}{n}\ESwapOfs1212{}{0,-0.25}{0,0.25}{0.5}\end{tikzpicture}}}

\def\tgrGD{\raise-7pt\hbox{\begin{tikzpicture}[scale=\GraphScale]
  \GraphVertices
  \Vertex[x=1.50,y=0.000,L=\relax]{1};
  \coordinate (2) at (3.00,0.000);
  \coordinate (3) at (0.000,0.000);
  \BlueEdges
  \LoopE(1)
  \LoopW(1)
\GraphEdgeSignS(1)(3){0.5}{n}\GraphEdgeSignS(1)(2){0.5}{n}\end{tikzpicture}}}

\def\tgrGE{\raise-7pt\hbox{\begin{tikzpicture}[scale=\GraphScale]
  \GraphVertices
  \Vertex[x=1.50,y=0.000,L=\relax]{1};
  \coordinate (2) at (3.00,0.000);
  \coordinate (3) at (0.000,0.000);
  \BlueEdges
  \LoopE(1)
  \LoopW(1)
\GraphEdgeSignS(1)(3){0.5}{n}\GraphEdgeSignS(1)(2){0.5}{n}\ESwapOfs1212{}{0,-0.25}{0,0.25}{0.5}\end{tikzpicture}}}

\def\tgrGF{\raise-7pt\hbox{\begin{tikzpicture}[scale=\GraphScale]
  \GraphVertices
  \Vertex[x=1.50,y=0.000,L=\relax]{1};
  \coordinate (2) at (3.00,0.000);
  \coordinate (3) at (0.000,0.000);
  \BlueEdges
  \LoopE(1)
  \LoopW(1)
\GraphEdgeSignS(1)(3){0.5}{n}\GraphEdgeSignS(1)(2){0.5}{n}\ESwapOfs1313{}{0,-0.25}{0,0.25}{0.5}\ESwapOfs1212{}{0,-0.25}{0,0.25}{0.5}\end{tikzpicture}}}

\def\tgrGG{\raise-7pt\hbox{\begin{tikzpicture}[scale=\GraphScale]
  \GraphVertices
  \Vertex[x=1.50,y=0.000,L=\relax]{1};
  \coordinate (2) at (3.00,0.000);
  \coordinate (3) at (0.000,0.000);
  \BlueEdges
  \LoopE(1)
  \LoopW(1)
\GraphEdgeSignS(1)(3){0.5}{n}\GraphEdgeSignS(1)(2){0.5}{n}\ESwapOfs1312{in=160,out=20}{0.2,0.3}{-0.2,0.3}{0.5}\end{tikzpicture}}}

\def\tgrGH{\raise-7pt\hbox{\begin{tikzpicture}[scale=\GraphScale]
  \GraphVertices
  \Vertex[x=1.50,y=0.000,L=\relax]{1};
  \coordinate (2) at (3.00,0.000);
  \coordinate (3) at (0.000,0.000);
  \BlueEdges
  \LoopE(1)
  \LoopW(1)
\GraphEdgeSignS(1)(3){0.5}{n}\GraphEdgeSignS(1)(2){0.5}{n}\EArrOfs1312{in=150,out=30}{0.1,0.29}{0,0.35}{0.5}\EArrOfs1213{in=-60,out=-60}{0,0.2}{0.3,-0.25}{0.5}\end{tikzpicture}}}

\def\tgrGA{\raise-3pt\hbox{\begin{tikzpicture}[scale=\GraphScale]
  \GraphVertices
  \Vertex[x=0.000,y=0.000,L=2]{1};
  \BlueEdges
\end{tikzpicture}}}

\usepackage{ifmtarg} 
\makeatletter
\newcommand{\TODO}[1]{\@ifmtarg{#1}{\textcolor{red}{\emph{\textbf{TODO}}}~}{\emph{\textcolor{red}{\textbf{TODO:}~#1}~}}}
\makeatother

\usepackage[makeroom]{cancel}
\usepackage[OT2,T1]{fontenc}
\DeclareSymbolFont{cyrletters}{OT2}{wncyr}{m}{n}
\DeclareMathSymbol{\Sha}{\mathalpha}{cyrletters}{"58}

\usetikzlibrary{%
	matrix,%
	calc,%
	arrows%
}
\hypersetup{
	colorlinks=true,
	linkcolor=black,
	filecolor=magenta,      
	urlcolor=black,
	citecolor=blue,
	pdftitle={Overleaf Example},
	pdfpagemode=FullScreen,
}

\newtheorem{thm}{Theorem}[section]

\newtheorem{lem}[thm]{Lemma}
\newtheorem{prop}[thm]{Proposition}

\theoremstyle{definition}
\newtheorem{ex}[thm]{Example}
\theoremstyle{definition}
\newtheorem{definition}[thm]{Definition}

\theoremstyle{remark}
\newtheorem*{rem}{Remark}

\newcommand{\Z}{\mathbb{Z}}

\newcommand{\Q}{\mathbb{Q}}

\newcommand{\F}{\mathbb{F}}

\title{Recovering Kodaira types from $\ell$-torsion on Elliptic curves}
\author[]{Naina Praveen}

\address{Department of Mathematics, University College London, Gower Street, London WC1E 6BT, UK}

\email{naina.praveen.24@ucl.ac.uk}

\begin{document}

		\begin{abstract}
      The classical N\'{e}ron-Ogg-Shafarevich criterion characterises good reduction of an elliptic curve $E$ over a $p$-adic field via the action of inertia on the $\ell$-adic Tate module. However, the action of inertia on $E[\ell]$ is not sufficient to distinguish between potentially good and multiplicative reduction, and the action on $T_{\ell}(E)$ is not sufficient to determine the Kodaira type. We remedy this situation by endowing $E[\ell]$ with a distance function that records the $p$-adic distances between the $x$-coordinates of the points. We show that, equipped with this additional structure, $E[\ell]$ determines the Kodaira type of the elliptic curve. In the case of residue characteristic $2$, we assume that $E$ does not have potentially good reduction of type $I_n^*$.
		\end{abstract}

    \maketitle

  \addtocontents{toc}{\protect\setcounter{tocdepth}{1}}
    \section{Introduction}\label{intro}

    The Kodaira type of an elliptic curve over a local field is a fundamental invariant of its reduction. Classically, it may be determined from a Weierstrass equation using Tate's algorithm. Meanwhile, the N\'{e}ron-Ogg-Shafarevich criterion determines good reduction via the action of inertia on the $\ell$-adic Tate module, but does not in general determine the Kodaira type. In fact, the inertia action on $E[\ell]$ does not even determine whether $E$ has potentially good or potentially multiplicative reduction.

    The purpose of this paper is to show that the $\ell$-torsion points $E[\ell]$ actually do contain a lot more information about the reduction of a curve. Let $K$ be a non-archimedean local field and let $E/K$ be an elliptic curve given by a Weierstrass equation of the form $y^2+a_1xy+a_3y=x^3+a_2x^2+a_4x+a_6$. Let $$\mathcal{X}_\ell := \{ x(P) \mid P \in E[\ell] \setminus \{\mathcal{O}_E\} \}$$ denote the set of distinct $x$-coordinates of the non-trivial $\ell$-torsion points of $E$, or equivalently, the roots of the $\ell$-division polynomial of $E$. Endowing $\mathcal X_\ell$ with the non-archimedean distance $d(x_1,x_2)=v_K(x_1-x_2)$, induced by the valuation of the local field $K$, we find that the arrangement of points in $\mathcal{X}_{\ell}$ already determine whether $E$ has potentially good or potentially multiplicative reduction.

    \begin{thm}[=Theorem \ref{thm:potential_red}]\label{thm:pot_intro}      
      Let $K$ be a non-archimedean local field of residue characteristic $p$. Let $E/K$ be an elliptic curve given by a Weierstrass equation, and let $\ell \ne 2,p$ be a prime. Then $E$ has potentially good reduction if and only if the points of $\mathcal{X}_{\ell}$ are equidistant. Equivalently, $E$ has potentially multiplicative reduction if and only if the points of $\mathcal{X}_{\ell}$ are not equidistant. 
    \end{thm}

    In fact, the distances between the points of $\mathcal X_\ell$ determine much more. Our main theorem shows that they recover the Kodaira type.

    \begin{thm}\label{thm:exact_kodaira_intro} 
      Let $K$ be a non-archimedean local field of residue characteristic $p$. Let $E/K$ be an elliptic curve given by a Weierstrass equation. For $p\ge 5$, the distances between points in $\mathcal{X}_{\ell}$ determine the Kodaira type of the special fibre. For $p=2,3$, assuming that $E$ does not have potentially good reduction of type $I_n^*$ (which can occur only when $p=2$), the distances along with the inertia module $E[\ell]$ determine the Kodaira type.
    \end{thm}

    It turns out that the distances between these $\ell$-torsion points alone do not always determine the Kodaira type for residue characteristics $2$ and $3$, and Theorem \ref{thm:exact_kodaira_intro} shows that we can recover it by requiring the action of inertia on $E[\ell]$ (see Example \ref{ex:2}). This parallels the classical Néron-Ogg-Shafarevich criterion, in that inertia still plays a role in detecting reduction, but only through its action on the finite group $E[\ell]$ rather than on the full $\ell$-adic Tate module. We also note that, in residue characteristics $2$ and $3$, Kraus \cite{Kraus} and Papadopoulos \cite{Papa} give descriptions of the Kodaira type in terms of valuations of invariants. Their methods, however, require first passing to a minimal Weierstrass model. By contrast, these $\ell$-torsion distances may be computed directly from an arbitrary Weierstrass model.

    The main tool in this paper is the notion of an $\ell$-torsion cluster picture, motivated by the cluster picture formalism of Dokchitser, Dokchitser, Maistret, and Morgan \cite{M2D2} (from now on referred to as ``standard cluster pictures'') for hyperelliptic curves. There, one studies the non-archimedean distances between Weierstrass points, obtaining a combinatorial object that encodes arithmetic information, including reduction types, minimal regular models, root numbers, and Galois representations. The resulting theory has found numerous applications, for instance to parity conjectures for abelian surfaces \cite{abeliansurface}, the construction of surjective Galois representations \cite{surjectivegalois}, local heights in quadratic Chabauty \cite{LocalHeights}, and questions concerning the Tate--Shafarevich group \cite{BoundSha,largesha}. 

    Before stating the main theorem, we briefly describe the $\ell$-torsion cluster pictures underlying our approach. For readers familiar with the cluster picture formalism of \cite{M2D2}, this construction may be viewed as an analogue obtained by replacing the roots of a polynomial $f(x)$ with those of the $\ell$-division polynomial of $E$. For other readers, it suffices to think of this as representing the non-archimedean distances on $\mathcal{X}_{\ell}$. More precisely, the $\ell$-torsion cluster picture records the distances $d(x_1,x_2)=v_K(x_i-x_j)$ for $x_i,x_j \in \mathcal{X}_{\ell}$, and can equivalently be viewed as the matrix of pairwise non-archimedean distances between these points. We postpone the precise definitions of clusters to Section \ref{sec:clusterpics}. For now, we illustrate the construction with a few examples.

    \begin{ex}\label{ex:intro1}
    Let $E: y^2 + y = x^3 + x^2 - x$ be an elliptic curve over $\Q_7$ and let $\ell=3$. Using Magma \cite{MAGMA}, one computes the $3$-division polynomial of $E$ as $3x^4 + 4x^3 - 6x^2 + 3x$. Writing $\alpha =\sqrt[3]{35}$, its roots, corresponding to the $x$-coordinates of the non-trivial $3$-torsion points, are
    $$x_1:=0, \quad x_2:=\frac{1}{9}(-4-2\alpha-\alpha^2), \quad x_3:=\frac{1}{18}(\alpha - 2)(4+\alpha - \sqrt{-3}\alpha), \quad x_4:=\frac{1}{18}(\alpha - 2)(4+\alpha + \sqrt{-3}\alpha).$$
    Let $M=\Q_7(E[3])/\Q_7$. A direct computation shows that $v_M(x_1-x_i)=0$ for $i=2,3,4$, while $v_M(x_i-x_j)=1$ for all distinct $i,j\in\{2,3,4\}$. Thus, $x_2,x_3$, and $x_4$ are pairwise closer to each other than any of them is to $x_1$. To compute the distances over $\Q_7$, we divide the valuations by $e_{M/\Q_7}$, which is $3$. One can represent this pictorially via a cluster picture as     
    \scalebox{1.3}{\clusterpicture
    \Root[D] {} {first} {r1};
    \Root[D] {1} {r1} {r2};
    \Root[D] {}{r2}{r3};
    \Root[D] {}{r3}{r4};
    \ClusterLD c1[][1/3] = (r2)(r3)(r4);
    \ClusterLD c2[][0] = (r1)(c1);
    \endclusterpicture} (on how to draw a cluster picture, see Section \ref{sec:clusterpics}). Then, from Theorem \ref{thm:pot_intro}, it follows that $E$ has potentially multiplicative reduction. Moreover, by Theorem \ref{thm:kodara_intro} (see below), scaling the second smallest distance by $\ell=3$ yields $1$, and so the curve has Kodaira type $I_1$ at $p=7$. 
    \end{ex}

    Example \ref{ex:intro1} is representative of the general pattern, and the $\ell$-torsion cluster picture of a curve with potentially multiplicative reduction always has the form drawn below. Readers familiar with standard cluster pictures may find the diagram a useful guide to Theorem \ref{thm:kodara_intro}, which gives an equivalent description purely in terms of the associated distances. For the unfamiliar reader, this schematic may be revisited after the definitions in Section \ref{sec:clusterpics}.
    
    \vspace{0.2cm}

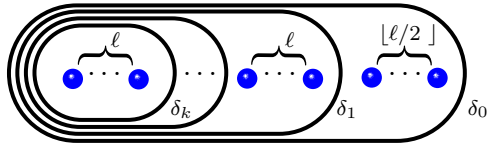
\begin{center}

    
  \tikzset {_hitfwvsm1/.code = {\pgfsetadditionalshadetransform{ \pgftransformshift{\pgfpoint{89.1 bp } { -128.7 bp }  }  \pgftransformscale{1.32 }  }}}
  \pgfdeclareradialshading{_xvezydr7i}{\pgfpoint{-72bp}{104bp}}{rgb(0bp)=(1,1,1);
  rgb(0bp)=(1,1,1);
  rgb(5.165018396718161bp)=(0,0,1);
  rgb(400bp)=(0,0,1)}
  
    
  \tikzset {_z1uwl8zrq/.code = {\pgfsetadditionalshadetransform{ \pgftransformshift{\pgfpoint{89.1 bp } { -128.7 bp }  }  \pgftransformscale{1.32 }  }}}
  \pgfdeclareradialshading{_5briogy58}{\pgfpoint{-72bp}{104bp}}{rgb(0bp)=(1,1,1);
  rgb(0bp)=(1,1,1);
  rgb(5.204080896718161bp)=(0,0,1);
  rgb(400bp)=(0,0,1)}
  
    
  \tikzset {_apxz9qq6r/.code = {\pgfsetadditionalshadetransform{ \pgftransformshift{\pgfpoint{89.1 bp } { -128.7 bp }  }  \pgftransformscale{1.32 }  }}}
  \pgfdeclareradialshading{_ccectyjl5}{\pgfpoint{-72bp}{104bp}}{rgb(0bp)=(1,1,1);
  rgb(0bp)=(1,1,1);
  rgb(5.165018396718161bp)=(0,0,1);
  rgb(400bp)=(0,0,1)}
  
    
  \tikzset {_2n5si7j4b/.code = {\pgfsetadditionalshadetransform{ \pgftransformshift{\pgfpoint{89.1 bp } { -128.7 bp }  }  \pgftransformscale{1.32 }  }}}
  \pgfdeclareradialshading{_zktknxges}{\pgfpoint{-72bp}{104bp}}{rgb(0bp)=(1,1,1);
  rgb(0bp)=(1,1,1);
  rgb(5.204080896718161bp)=(0,0,1);
  rgb(400bp)=(0,0,1)}
  
    
  \tikzset {_30idcifjt/.code = {\pgfsetadditionalshadetransform{ \pgftransformshift{\pgfpoint{89.1 bp } { -128.7 bp }  }  \pgftransformscale{1.32 }  }}}
  \pgfdeclareradialshading{_f8dtnjso9}{\pgfpoint{-72bp}{104bp}}{rgb(0bp)=(1,1,1);
  rgb(0bp)=(1,1,1);
  rgb(5.165018396718161bp)=(0,0,1);
  rgb(400bp)=(0,0,1)}
  
    
  \tikzset {_1mve3ghwt/.code = {\pgfsetadditionalshadetransform{ \pgftransformshift{\pgfpoint{89.1 bp } { -128.7 bp }  }  \pgftransformscale{1.32 }  }}}
  \pgfdeclareradialshading{_2xc62b979}{\pgfpoint{-72bp}{104bp}}{rgb(0bp)=(1,1,1);
  rgb(0bp)=(1,1,1);
  rgb(5.204080896718161bp)=(0,0,1);
  rgb(400bp)=(0,0,1)}
  \tikzset{every picture/.style={line width=0.75pt}} 
  
  \begin{tikzpicture}[x=0.75pt,y=0.75pt,yscale=-1,xscale=1]
  
  \draw  [line width=1.5]  (10.14,39.14) .. controls (10.14,20.37) and (25.37,5.14) .. (44.14,5.14) -- (204.93,5.14) .. controls (223.71,5.14) and (238.93,20.37) .. (238.93,39.14) -- (238.93,39.14) .. controls (238.93,57.92) and (223.71,73.14) .. (204.93,73.14) -- (44.14,73.14) .. controls (25.37,73.14) and (10.14,57.92) .. (10.14,39.14) -- cycle ;
  \draw  [line width=1.5]  (22.95,38.93) .. controls (22.95,25.85) and (33.55,15.24) .. (46.64,15.24) -- (69.66,15.24) .. controls (82.74,15.24) and (93.34,25.85) .. (93.34,38.93) -- (93.34,38.93) .. controls (93.34,52.01) and (82.74,62.61) .. (69.66,62.61) -- (46.64,62.61) .. controls (33.55,62.61) and (22.95,52.01) .. (22.95,38.93) -- cycle ;
  \draw  [line width=1.5]  (18.64,38.89) .. controls (18.64,23.84) and (30.84,11.64) .. (45.89,11.64) -- (91.95,11.64) .. controls (107,11.64) and (119.2,23.84) .. (119.2,38.89) -- (119.2,38.89) .. controls (119.2,53.94) and (107,66.14) .. (91.95,66.14) -- (45.89,66.14) .. controls (30.84,66.14) and (18.64,53.94) .. (18.64,38.89) -- cycle ;
  \draw  [line width=1.5]  (14.34,38.9) .. controls (14.34,21.95) and (28.08,8.21) .. (45.03,8.21) -- (145.91,8.21) .. controls (162.86,8.21) and (176.6,21.95) .. (176.6,38.9) -- (176.6,38.9) .. controls (176.6,55.85) and (162.86,69.59) .. (145.91,69.59) -- (45.03,69.59) .. controls (28.08,69.59) and (14.34,55.85) .. (14.34,38.9) -- cycle ;
  \draw  [draw opacity=0][shading=_xvezydr7i,_hitfwvsm1][line width=0.75]  (80.31,42.25) .. controls (80.29,45.07) and (77.99,47.35) .. (75.17,47.33) .. controls (72.35,47.31) and (70.08,45.02) .. (70.1,42.2) .. controls (70.11,39.38) and (72.41,37.1) .. (75.23,37.12) .. controls (78.05,37.14) and (80.32,39.43) .. (80.31,42.25) -- cycle ;
  \draw  [draw opacity=0][shading=_5briogy58,_z1uwl8zrq][line width=0.75]  (47.31,42.04) .. controls (47.29,44.86) and (44.99,47.13) .. (42.17,47.12) .. controls (39.35,47.1) and (37.08,44.8) .. (37.1,41.98) .. controls (37.11,39.17) and (39.41,36.89) .. (42.23,36.91) .. controls (45.05,36.92) and (47.32,39.22) .. (47.31,42.04) -- cycle ;
  \draw  [draw opacity=0][shading=_ccectyjl5,_apxz9qq6r][line width=0.75]  (167.81,42.06) .. controls (167.79,44.88) and (165.49,47.15) .. (162.67,47.14) .. controls (159.85,47.12) and (157.58,44.82) .. (157.6,42) .. controls (157.61,39.18) and (159.91,36.91) .. (162.73,36.93) .. controls (165.55,36.94) and (167.82,39.24) .. (167.81,42.06) -- cycle ;
  \draw  [draw opacity=0][shading=_zktknxges,_2n5si7j4b][line width=0.75]  (134.81,41.85) .. controls (134.79,44.67) and (132.49,46.94) .. (129.67,46.92) .. controls (126.85,46.91) and (124.58,44.61) .. (124.6,41.79) .. controls (124.61,38.97) and (126.91,36.7) .. (129.73,36.72) .. controls (132.55,36.73) and (134.82,39.03) .. (134.81,41.85) -- cycle ;
  \draw  [draw opacity=0][shading=_f8dtnjso9,_30idcifjt][line width=0.75]  (230.31,41.56) .. controls (230.29,44.38) and (227.99,46.65) .. (225.17,46.64) .. controls (222.35,46.62) and (220.08,44.32) .. (220.1,41.5) .. controls (220.11,38.68) and (222.41,36.41) .. (225.23,36.43) .. controls (228.05,36.44) and (230.32,38.74) .. (230.31,41.56) -- cycle ;
  \draw  [draw opacity=0][shading=_2xc62b979,_1mve3ghwt][line width=0.75]  (197.31,41.35) .. controls (197.29,44.17) and (194.99,46.44) .. (192.17,46.42) .. controls (189.35,46.41) and (187.08,44.11) .. (187.1,41.29) .. controls (187.11,38.47) and (189.41,36.2) .. (192.23,36.22) .. controls (195.05,36.23) and (197.32,38.53) .. (197.31,41.35) -- cycle ;
  
  \draw (59.75,17.09) node [anchor=north west][inner sep=0.75pt]  [font=\scriptsize]  {$\ell $};
  \draw (96,37) node [anchor=north west][inner sep=0.75pt]    {$\dotsc $};
  \draw (89.71,51) node [anchor=north west][inner sep=0.75pt]  [font=\scriptsize]  {$\delta _{k}$};
  \draw (173.13,51) node [anchor=north west][inner sep=0.75pt]  [font=\scriptsize,rotate=-0.87]  {$\delta _{1}$};
  \draw (239.01,51) node [anchor=north west][inner sep=0.75pt]  [font=\scriptsize]  {$\delta _{0}$};
  \draw (74,22) node [anchor=north west][inner sep=0.75pt]  [font=\huge,rotate=-90.68]  {$\{$};
  \draw (48,37) node [anchor=north west][inner sep=0.75pt]    {$\dotsc $};
  \draw (147.25,16.9) node [anchor=north west][inner sep=0.75pt]  [font=\scriptsize]  {$\ell $};
  \draw (161,22) node [anchor=north west][inner sep=0.75pt]  [font=\huge,rotate=-90.68]  {$\{$};
  \draw (136,37) node [anchor=north west][inner sep=0.75pt]    {$\dotsc $};
  \draw (194.25,13) node [anchor=north west][inner sep=0.75pt]  [font=\scriptsize]  {$\lfloor \ell /2\ \rfloor $};
  \draw (224,22) node [anchor=north west][inner sep=0.75pt]  [font=\huge,rotate=-90.68]  {$\{$};
  \draw (199,37) node [anchor=north west][inner sep=0.75pt]    {$\dotsc $};

  \end{tikzpicture}

\end{center}
\vspace{0.2cm}
Here, $k=\lfloor \ell/2 \rfloor$, and this description can be obtained from the proof of Theorem \ref{thm:pot_intro}. One may observe that the roots of the $\ell$-division polynomial group into a sequence of nested clusters. The numbers $\delta_{i}$ represent the relative depth (see Definition \ref{defn:relative depth}). As we will show in the proof of Theorem \ref{thm:kodara_intro}, $\delta_{i}=n/\ell$ for all $1 \le i \le  k$ (where $n=-v_K(j(E))$) and determine the resulting Kodaira type. The following theorem makes this determination explicit.
    
    \begin{thm}\label{thm:kodara_intro}
      Let $K$ be a non-archimedean local field with residue characteristic $p$. Let $E/K$ be an elliptic curve given by a Weierstrass equation $y^2+a_1xy+a_3y= x^3 +a_2x^2 + a_4x + a_6$. Let $\mathcal{X}_{\ell}:= \{ x(P) \mid P \in E[\ell] \setminus \{\mathcal{O}_E\} \}$ denote the set of roots of the $\ell$-division polynomial of $E$, endowed with the non-archimedean distance $d(x_1,x_2)=v_K(x_1-x_2)$. 

      \begin{enumerate}[label=(\alph*)]
        \item Suppose $E$ has potentially good reduction. Then, the points in $\mathcal{X}_{\ell}$ of $E$ are all equidistant with distance $d$. Assume the special fibre is not of type $I_n^*$, which could only possibly occur at $p=2$. The Kodaira type of its special fibre is determined by $d$ and $\eta$ where
        $$\eta=\left\{\begin{array}{ll}
          2 & \text { if $v_K(N)=0$ }\\
          v_K(N) & \text { otherwise,}
          \end{array}\right.
          $$
          and $N$ denotes the conductor of $E$, as follows:

           \vspace{0.1in}
       \begin{center}
          \renewcommand{\arraystretch}{1.2}
          \begin{tabular}{|c|c|c|c|c|c|c|c|c|}
            
            \hline
            Kodaira type   & $I_0$ & $II$ & $III$ & $IV$ & $I_0^*$  & $IV^*$ & $III^*$ & $II^*$ \\
            \hline
            $6d -\eta+2 \mod 12 $ & $0$ & $2$ & $3$ & $4$ & $6$ & $8$ & $9$ & $10$ \\
            \hline
          \end{tabular}
      \end{center}
      \vspace{0.1in}

        \item  Suppose $E$ has potentially multiplicative reduction. Then, the points in $\mathcal{X}_{\ell}$ are not equidistant. Let $d_0$ and $d_1$ be the two smallest (distinct) distances, with $d_0<d_1$, and let $\delta_1=d_1-d_0$.         

    \begin{enumerate}[label=(\roman*)]
      \item $(p \ne 2)$:  
      The Kodaira type can be inferred from $d_0$ and $\delta_1$ as below:
      
      \vspace{0.1in}
           \begin{center}
              \renewcommand{\arraystretch}{1.2}
              \begin{tabular}{|c|c|c|}
                
                \hline
                Kodaira type   & $I_n$ & $I_n^*$ \\
                \hline
                $d_0 \mod 2 $ & $0$ & $1$ \\
                \hline
                $\delta_1$ & $n/\ell$ & $n/\ell$ \\
                \hline
              \end{tabular}
          \end{center}
          \vspace{0.1in}

          \item $(p=2)$: If the curve has multiplicative reduction, then $\delta_1$ is $n/\ell$, yielding the integer $n$ for type $I_n$. If the curve has additive reduction, then let $s_{L/K}$ be the valuation of the different of the extension $L/K$ where $L$ is the minimal extension over which $E$ acquires semistable reduction. If $\delta_1=n/\ell$, then the Kodaira type at $p=2$ is given by $I_{n+4(s_{L/K}-1)}^*$. 
    \end{enumerate}
      \end{enumerate}     
    \end{thm}

    Since the inertia action on $E[\ell]$ determines the conductor \cite[\S 10, IV]{SilII} and the extension $L$ (see \S \ref{subsec:differentiating}), it is evident that Theorem \ref{thm:exact_kodaira_intro} follows from Theorem \ref{thm:kodara_intro}.

    \begin{ex}\label{ex:2}
      Consider the elliptic curve given by $y^2=x^3+x^2+95x-1057$ over $\Q_2$. One computes $d_0=3$ and $\delta_1=4/3$, i.e., its $3$-torsion cluster picture is given by \scalebox{1.3}{
        \clusterpicture
        \Root[D] {} {first} {r1};
        \Root[D] {1} {r1} {r2};
        \Root[D] {}{r2}{r3};
        \Root[D] {}{r3}{r4};
        \ClusterLD c1[][4/3] = (r2)(r3)(r4);
        \ClusterLD c2[][3] = (r1)(c1);
        \endclusterpicture
      }.
      Since the torsion points are not equidistant, the curve has potentially multiplicative reduction. Let $M=\Q_2(E[3])$. As $e_{M/\Q_2}=6$, the inertia action on $E[3]$ contains an element of order $2$. Since every non-trivial unipotent element of $\mathrm{GL}_2(\F_3)$ has order $3$, inertia is not purely unipotent, and hence $E$ has Kodaira type $I_n^*$ for some $n$. The quadratic extensions $L$ over which $E$ acquires multiplicative reduction correspond to the index-$2$ subfields of $\Q_2(E[3])$, which are $\Q_2(\sqrt{-2})$ and $\Q_2(\sqrt{-10})$, corresponding to the split and non-split multiplicative cases. Both yield the same valuation of the different  $s_{L/\Q_2} = 3$. So, from Theorem \ref{thm:kodara_intro}, $E/\Q_2$ has Kodaira type $I_{12}^*$.

    \end{ex}

    \begin{ex}\label{ex:3}
      Note, we assume that at $p=2$, the curve is not potentially good of type $I_n^*$. One might ask if incorporating further representation-theoretic data could resolve this. Consider the elliptic curves given by LMFDB labels $100048.g1$ and $100048.n1$ over $\mathbb{Q}_2$, having reduction types $I_4^*$ and $II^*$ respectively \cite{lmfdb}. They have identical $\ell$-torsion cluster pictures and furthermore, both curves acquire good reduction over  $\mathbb{Q}_2(\sqrt{-1})$ with a $C_2$ inertia action. Consequently, for $\ell \neq 2$, their $\ell$-adic Tate modules are isomorphic over $\Q_2^{nr}$. 
    \end{ex}

    \subsection*{Outline of paper}
    In Section \ref{sec:clusterpics}, we introduce $\ell$-torsion cluster pictures using the framework of standard cluster pictures. In Section \ref{sec:base cases}, we prove Theorem \ref{thm:pot_intro}, the result on potential reduction type, by first determining the cluster picture in the semistable cases of good and multiplicative reduction, and then analysing how cluster pictures vary under changes of Weierstrass models to complete the proof. In Section \ref{sec:absolute_depths}, we compute the explicit distances on $\ell$-torsion cluster pictures for minimal models, and in Section \ref{sec:classification} we assemble these results to determine the Kodaira type, thereby proving Theorem \ref{thm:kodara_intro}.

    \section{Notation}\label{sec:notation}

Throughout this paper, we adopt the following notation for fields, elliptic curves, and cluster pictures.

\subsubsection*{Fields and Elliptic curves}

\begin{center}
\begin{tabular}{p{1.1cm} p{12cm}}
  $K$ &  A non-archimedean local field, with residue field $k$.\\

  $p$ & The residue characteristic of $K$\\
  $\ell$ & A prime distinct from the residue characteristic $p$ and $\ell \neq 2$. \\

\end{tabular}

\begin{tabular}{p{1.1cm} p{12cm}}

  $v_K$ & The discrete valuation on $K$, scaled such that $v_K(K^\times) = \mathbb{Z}$. \\

  $E/K$ & An elliptic curve defined over $K$. \\  
  $E[n]$ & The group of $n$-torsion points of $E(\overline{K})$ for any $n \in \Z_{> 0}$. \\

  $\Delta_{\min}$ & The minimal discriminant of $E/K$ with respect to $\mathcal{O}_K$. \\
  $v_K(N)$ & The valuation of the conductor of $E/K$. \\
  
  $e_{L/K}$ & The ramification index of an extension $L/K$. \\
  $s_{L/K}$ & Valuation of the different of an extension $\mathcal{O}_L/\mathcal{O}_K$, defined by Hilbert's formula $s_{L/K} = \sum_{i=0}^\infty (|H_i| - 1)$, where $H_i$ are the higher ramification groups.
\end{tabular}
\end{center}

\subsubsection*{Torsion and Cluster Pictures}

\begin{center}
  \begin{tabular}{p{1.1cm} p{12cm}} 
    $\mathcal{X}_{\ell}$ & The set of distinct $x$-coordinates of the non-zero $\ell$-torsion points of $E$, having cardinality $(\ell^2 - 1)/2$. \\

    $\mathfrak{s}$ & A \textit{cluster}, defined as a non-empty subset $\mathfrak{s} \subseteq \mathcal{X}_\ell$ of the form $\mathfrak{s} = D \cap \mathcal{X}_\ell$ for some disc $D = \{x \in \overline{K} \mid \frac{1}{e}v_L(x - z) \ge d\}$, some $z \in \overline{K}$, and $d\in \mathbb{Q}$. \\

    $d_{\mathfrak{s}}$ & The \textit{absolute depth} of a cluster $\mathfrak{s}$, defined as $\min_{x_i, x_j \in \mathfrak{s}} \frac{1}{e} v_L(x_i - x_j)$.\\

    $\delta_{\mathfrak{s}}$ & The \textit{relative depth} of a subcluster $\mathfrak{s} \subsetneq \mathcal{X}_{\ell}$. See Definition \ref{defn:relative depth}.\\
    
  \end{tabular}
\end{center}

\addtocontents{toc}{\protect\setcounter{tocdepth}{0}}
\section*{Acknowledgements}
The author would like to express their sincere gratitude to Vladimir Dokchitser for suggesting this problem and for his  support throughout this project. This work was supported by the London School of Geometry and Number Theory Centre for Doctoral Training (CDT), a joint venture between University College London, Imperial College London and King's College London.

    \addtocontents{toc}{\protect\setcounter{tocdepth}{2}}
    \section{\texorpdfstring{$\ell$}{}-torsion cluster pictures and results}\label{sec:clusterpics}

    Throughout, let $K$ be a non-archimedean local field  of residue characteristic $p$, with normalised valuation $v_K$. Let $E/K$ be an elliptic curve with Weierstrass equation $y^2+a_1xy+a_3y= x^3 +a_2x^2 + a_4x + a_6$, and fix a prime $\ell \neq p,2$. Let $M/K$ be the $\ell$-division field $M:=K(E[\ell])$, and write $e = [M:K^{\mathrm{unr}}]$ for the ramification index, where $K^{unr}$ is the maximal unramified extension of $K$. Let $v_M$ denote the normalised valuation on $M$. 
    
    We study the structure of the $\ell$-torsion subgroup $E[\ell] \subset E(\overline{K})$ in terms of clusters. The involution $[-1]$ acts on $E[\ell]$ by negating the $y$-coordinate, so the non-zero $\ell$-torsion points occur in pairs $\pm P$ with the same $x$-coordinate. Let
    $$\mathcal{X}_\ell := \{ x(P) \mid P \in E[\ell] \setminus \{\mathcal{O}_E\} \}.$$
    Note that $|\mathcal{X}_\ell| = (\ell^2 - 1)/2$. We now define the $\ell$-torsion cluster picture of $E$ by adapting the standard cluster formalism of \cite{M2D2} to the set $\mathcal{X}_\ell$.

\begin{definition}
An \textit{($\ell$-torsion) cluster} is a non-empty subset $\mathfrak{s} \subseteq \mathcal{X}_\ell$ of the form $\mathfrak{s} = D \cap \mathcal{X}_\ell$ for some disc $D = \{x \in \overline{K} \mid \frac{1}{e}v_M(x - z) \ge d\}$, for some $z \in \overline{K}$ and $d\in \mathbb{Q}$. $\Sigma_{\ell}$ is the collection of all clusters of $\mathcal{X}_\ell$. By the ultrametric property of valuations, any two clusters in $\Sigma_{\ell}$ are either disjoint or contained in one another. $\Sigma_{\ell}$ includes the maximal set $\mathcal{X}_\ell$.
\end{definition}

\begin{definition} If $|\mathfrak{s}| >1$, we define its \textit{absolute depth} as $$\min_{\substack{x_i, x_j \in \mathfrak{s} \\ x_i \neq x_j}} \frac{1}{e} v_M(x_i - x_j).$$
\end{definition}

\begin{definition}\label{defn:relative depth}
  If $\mathfrak{s} \neq \mathcal{X}_\ell$, its \textit{parent cluster} $P(\mathfrak{s})$ is defined as the smallest cluster in $\Sigma$ strictly containing $\mathfrak{s}$. A cluster $\mathfrak{s}'$ is a \textit{child} of $\mathfrak{s}$ if $\mathfrak{s}'$ is a maximal subcluster of $\mathfrak{s}$. The \textit{relative depth} of a proper cluster $\mathfrak{s}$ is the distance to its parent, $\delta_{\mathfrak{s}} := d_{\mathfrak{s}} - d_{P(\mathfrak{s})}.$ We set $\delta_{\mathcal{X}_\ell} := d_{\mathcal{X}_{\ell}}$.
  \end{definition}

\begin{definition}
An \textit{$\ell$-torsion cluster picture} is a triplet $(\Sigma_{\ell}, \mathcal{X}_\ell, \delta)$ where:
\begin{enumerate}
    \item $\mathcal{X}_\ell := \{ x(P) \mid P \in E[\ell] \setminus \{\mathcal{O}_E\} \}$ is the set of distinct $x$-coordinates of the non-trivial $\ell$-torsion points.
    \item $\Sigma_{\ell}$ is the collection of all clusters of $\mathcal{X}_\ell$. 
    \item $\delta: \Sigma_{\ell} \to \mathbb{Q} \cup \{\infty\}$ is a depth function that assigns to each cluster $\mathfrak{s} \in \Sigma_{\ell}$ of size $|\mathfrak{s}| > 1$ its relative depth $\delta_{\mathfrak{s}}$.
\end{enumerate}
\end{definition}

On how to pictorially represent a cluster picture, we refer the reader to \cite{M2D2}. We define the \textit{shape} of an $\ell$-torsion cluster picture as the pair $(\mathcal{X}_{\ell}, \Sigma_{\ell})$. In the context of $\ell$-torsion on elliptic curves, the shape of these cluster pictures is limited. We will show in the proof of Theorem \ref{thm:potential_red} that the set $\mathcal{X}_\ell$ will always partition into one of two specific shapes:
    
    \begin{enumerate}
        \item \textit{Good cluster picture:} 
        All distinct $x$-coordinates in $\mathcal{X}_\ell$ are equidistant from one another. In this case, $\mathcal{X}_\ell$ itself is the only cluster of size $>1$, and $|\Sigma_\ell|=1$. 

        \begin{ex}
         $\ell=3$: \scalebox{1.5}{\clusterpicture
          \Root[D] {} {first} {r1};
          \Root[D] {} {r1} {r2};
          \Root[D] {}{r2}{r3};
          \Root[D] {}{r3}{r4};
          \ClusterLD c1[\mathcal{X}_3][] = (r1)(r2)(r3)(r4);
          \endclusterpicture}  \quad  $\ell=5$: \scalebox{1.5}{
            \clusterpicture
            \Root[D] {} {first} {r1};
            \Root[D] {} {r1} {r2};
            \Root[D] {}{r2}{r3};
            \Root[D] {}{r3}{r4};
            \Root[D] {}{r4}{r5};
            \Root[D] {}{r5}{r6};
            \Root[D] {}{r6}{r7};
            \Root[D] {}{r7}{r8};
            \Root[D] {}{r8}{r9};
            \Root[D] {}{r9}{r10};
            \Root[D] {}{r10}{r11};
            \Root[D] {}{r11}{r12};
            \ClusterLD c1[\mathcal{X}_5][] = (r1)(r2)(r3)(r4)(r5)(r6)(r7)(r8)(r9)(r10)(r11)(r12);
            \endclusterpicture}
        \end{ex}
        
        \item \textit{Multiplicative cluster picture:} 
        The $x$-coordinates partition into a strictly nested sequence of $k = \lfloor \ell/2 \rfloor$ proper subclusters. There exists a single chain of strictly contained clusters:
        $$
        \mathfrak{s}_k \subsetneq \mathfrak{s}_{k-1} \subsetneq \dots \subsetneq \mathfrak{s}_1 \subsetneq \mathcal{X}_\ell
        $$
        where the innermost cluster has size $|\mathfrak{s}_k| = \ell$, and each subsequent layer contains $\ell$ additional points, such that $|\mathfrak{s}_i \setminus \mathfrak{s}_{i+1}| = \ell$ for $1 \le i < k$. The complement in the maximal cluster contains $|\mathcal{X}_\ell \setminus \mathfrak{s}_1| = \lfloor \ell/2 \rfloor$ points. The cardinality of $\Sigma_{\ell}$ is $k+1$.

        \begin{ex}
          $\ell=3$:\scalebox{1.5}{
            \clusterpicture
            \Root[D] {} {first} {r1};
            \Root[D] {1} {r1} {r2};
            \Root[D] {}{r2}{r3};
            \Root[D] {}{r3}{r4};
            \ClusterLD c1[\mathfrak{s}_1][] = (r2)(r3)(r4);
            \ClusterLD c2[\mathcal{X}_3][] = (r1)(c1);
            \endclusterpicture
          } \quad $\ell=5:$ \scalebox{1.5}{
            \clusterpicture
            \Root[D] {} {first} {r1};
            \Root[D] {} {r1} {r2};
            \Root[D] {1}{r2}{r3};
            \Root[D] {}{r3}{r4};
            \Root[D] {}{r4}{r5};
            \Root[D] {}{r5}{r6};
            \Root[D] {}{r6}{r7};
            \Root[D] {1}{r7}{r8};
            \Root[D] {}{r8}{r9};
            \Root[D] {}{r9}{r10};
            \Root[D] {}{r10}{r11};
            \Root[D] {}{r11}{r12};
            \ClusterLD c1[\mathfrak{s}_2][] = (r8)(r9)(r10)(r11)(r12);
            \ClusterLD c2[\mathfrak{s}_1][] = (c1)(r3)(r4)(r5)(r6)(r7);
            \ClusterLD c3[\mathcal{X}_5][] = (r1)(r2)(c1)(c2);
            \endclusterpicture
            }
        \end{ex}
        
    \end{enumerate}

\section{Potential reduction type via $\ell$-torsion cluster pictures}\label{sec:base cases}

In this section, we will prove the following  theorem on potential reduction type.
 
\begin{thm} \label{thm:potential_red}
Let $K$ be a non-archimedean local field of any residue characteristic $p$, and let $E/K$ be an elliptic curve given by a Weierstrass model $y^2+a_1xy+a_3y=x^3+a_2x^2+a_4x+a_6$. Let $\ell \ne 2,p$ be a prime. The shape of the $\ell$-torsion cluster picture determines whether $E$ has potentially good or potentially multiplicative reduction, independent of the choice of model. In particular,
\begin{enumerate}[label=(\roman*)]
    \item $E$ has potentially good reduction if and only if it has a good $\ell$-torsion cluster picture shape.
    \item $E$ has potentially multiplicative reduction if and only if it has a multiplicative $\ell$-torsion cluster picture shape.
\end{enumerate}
\end{thm}

    By the semistable reduction theorem, an elliptic curve attains either good or multiplicative reduction after a finite extension. We first compute the $\ell$-torsion distances in these two semistable cases, assuming a minimal Weierstrass model. We then complete the proof of Theorem \ref{thm:potential_red} by understanding how the valuations change under a transformation of Weierstrass equations, showing that the resulting shape of the $\ell$-torsion cluster picture determines the potential reduction type.
    
    \begin{prop}\label{prop:base_cases}
      Let $E$ be an elliptic curve over $K$ given by a minimal Weierstrass equation, and let $\ell \ne p,2$ be a prime.
      \begin{enumerate}[label=(\roman*)]
      \item \textit{Good reduction:} If $E$ has good reduction, then $\mathcal{X}_\ell$ has a single cluster and $d_{\mathcal{X}_\ell} = 0$.
      
      \item \textit{Multiplicative reduction:} If $E$ has multiplicative reduction (split or non-split) of type $I_n$ with $n \ge 1$, then the $\ell$-torsion cluster picture is multiplicative, with $d_{\mathcal{X}_\ell} = 0$, and $\delta_{\mathfrak{s}_i} = n/\ell$ for each subcluster $\mathfrak{s}_i \in \Sigma_{\ell}$ with $\mathfrak{s}_i \ne \mathcal{X}_{\ell}$.
      \end{enumerate}
      \end{prop}

    \begin{proof}
   
        \textit{Case I (Good Reduction):}        
        Assume $E$ has good reduction at $p$ and is given by a minimal Weierstrass equation. Because $\ell \neq p$, the reduction map $\pi: E[\ell] \hookrightarrow \widetilde{E}(k)$ is injective.
        
        Let $x_i \neq x_j$ be distinct coordinates in $\mathcal{X}_\ell$, lifting to points $P, Q \in E[\ell]$ such that $P \neq \pm Q$. Suppose for contradiction that $v_K(x_i - x_j) > 0$, meaning $\widetilde{x}_i = \widetilde{x}_j$ in the residue field $k$. Because the reduced points lie on $\widetilde{E}$, sharing an $x$-coordinate forces $\widetilde{P} = \pm \widetilde{Q}$. By the injectivity of $\pi$, this implies $P = \pm Q$, which contradicts $x_i \neq x_j$. Therefore, $v_K(x_i - x_j) = 0$ for all distinct pairs, and $\mathcal{X}_\ell$ forms a single cluster of depth $0$.

        \vspace{0.1in}
        \textit{Case II (Multiplicative reduction):} If $E/K$ has non-split multiplicative reduction, it acquires split multiplicative reduction over an unramified quadratic extension $K'/K$. Since $e_{K'/K}=1$, the valuations of the $\ell$-torsion points remain unchanged. So, we may assume that $E$ has split multiplicative reduction over $K$ and is of type $I_n$. By the theory of Tate curves (\cite[Theorem 3.3, V]{SilII}, \cite{Roquette}), there is an analytic isomorphism $E \cong E_q$ for some $q \in K^\times$ with $v_K(q) = n > 0$. The isomorphism is given by the surjective homomorphism
        $$
        \phi: \overline{K}^{\times} / q^{\mathbb{Z}} \longrightarrow E_q(\overline{K}), \qquad
        u \longmapsto 
        \begin{cases}
            (X(u,q), Y(u,q)) & \text{if } u \notin q^{\mathbb{Z}}, \\
            \mathcal{O}_E & \text{if } u \in q^{\mathbb{Z}}.
        \end{cases}
        $$
        
        To compute the valuation $v_K(x_i - x_j)$ for distinct $x_i, x_j \in \mathcal{X}_\ell$, we evaluate the difference $X(u_i, q) - X(u_j, q)$ for their corresponding parameters $u_i, u_j \in E_q[\ell]$. We first recall the identity expressing this difference in terms of the $\theta$-function (\cite[Proposition 3.2, V]{SilII}):
        \begin{equation} \label{eq:x-difference-theta}
            X(u_i, q) - X(u_j, q) = -\frac{u_j \, \theta(u_i u_j, q) \theta(u_i u_j^{-1}, q)}{\theta(u_i, q)^2 \theta(u_j, q)^2}, \tag{\dag}
        \end{equation}
        where $$\theta(u, q) := (1 - u) \prod_{k \ge 1} \frac{(1 - q^k u)(1 - q^k u^{-1})}{(1 - q^k)^2}.$$
        We remark that Silverman proves (\ref{eq:x-difference-theta}) only over $p$-adic fields, but the same holds over non-archimedean local fields by comparing divisors and leading Laurent coefficients of the two $q$-periodic functions \cite[\S2, \S3]{Roquette}. 
        
        Because points in $E[\ell]$ share an $x$-coordinate if and only if they are inverses, the parameters $u \in E_q[\ell]$ share an $X$-coordinate if and only if $u_i = u_j^{-1}$. Therefore, we can uniquely identify every distinct $x$-coordinate in $\mathcal{X}_\ell$ by restricting our representatives to:
        $$
        u = \zeta_\ell^r q^{s/\ell} \quad \text{for } \quad r \in \{0, \dots, \ell-1\} \quad \text{and} \quad s \in \{0, 1, \dots, \lfloor \ell/2 \rfloor\}.
        $$
        (Excluding $u=1$, which corresponds to $\mathcal{O}_E$).
        
        Let $x_i, x_j \in \mathcal{X}_\ell$ be distinct coordinates corresponding to $u_i = \zeta_{\ell}^{r_i} q^{s_i/\ell}$ and $u_j = \zeta_{\ell}^{r_j} q^{s_j/\ell}$. Without loss of generality, assume $s_j \le s_i$, so $v_K(u_j) \le v_K(u_i)$. 
        
        To evaluate equation (\ref{eq:x-difference-theta}), we first understand the valuation of $\theta(w, q)$ for any generic value $w$ of the form $\zeta_\ell^r q^{s/\ell}$, where $0 \le s < \ell$, $r \in \mathbb{Z}$, and $w \neq 1$. 
      Looking at the factors in $\theta(w, q)$, we have
      \begin{itemize}
        \item $v_K(1 - w) = 0$. (If $s > 0$, then $v_K(w) > 0$. If $s = 0$, then $w = \zeta_\ell^r \neq 1$, and $v_K(1 - \zeta_\ell^r) = 0$ since $p \neq \ell$).
        \item $v_K(1 - q^k w) = 0$ because $v_K(q^k w) = kn + (sn/\ell) > 0$ for all $k \ge 1$.
        \item $v_K(1 - q^k w^{-1}) = 0$ because $v_K(q^k w^{-1}) = n(k -(s/\ell)) > 0$ since $k \ge 1$ and $s/\ell < 1$.
        \item $v_K(1 - q^k) = 0$ similarly.
      \end{itemize}
      Thus, for any such $w$, the infinite product defining $\theta(w,q)$ evaluates to a unit, and so $v_K(\theta(w, q)) = 0$.
        
        We now check the four arguments passed to the $\theta$-functions in (\ref{eq:x-difference-theta}), namely $u_i$, $u_j$, $u_i u_j$, and $u_i u_j^{-1}$. Because $u_i$ and $u_j$ represent distinct $x$-coordinates, we know $u_i \neq 1$, $u_j \neq 1$, and $u_i \neq u_j^{\pm 1}$. 
        The exponent of $q$ for the product $u_i u_j$ is $(s_i + s_j)/\ell \le 2\lfloor \ell/2 \rfloor/\ell \le (\ell-1)/\ell < 1$. 
        The exponent of $q$ for the quotient $u_i u_j^{-1}$ is $(s_i - s_j)/\ell \le \lfloor \ell/2 \rfloor/\ell < 1$.
        Because both exponents are strictly less than $1$ and neither expression equals $1$, all the four $\theta$-function terms in (\ref{eq:x-difference-theta}) have valuation $0$.
        
        Applying this to (\ref{eq:x-difference-theta}) yields 
        $$
        v_K(x_i - x_j) = v_K(u_j) = \frac{s_j}{\ell} n = \min(s_i, s_j) \frac{n}{\ell}.
        $$
        As the distance between any two $x$-coordinates is given by the minimum of their respective $s$ parameters, the elements of $\mathcal{X}_\ell$ partition into a strictly nested sequence of clusters. The outermost cluster (where $s=0$) has depth $0$, and each subsequent subcluster (incrementing $s$) adds $n/\ell$ to the depth, forming the $\lfloor \ell/2 \rfloor$ nested subclusters as claimed.

    \end{proof}

    \begin{lem} \label{lem:transformation}
    If Weierstrass models $E$ and $E'$ over $K$ are related by $x = u^2 x' + r$ for some $u \in K^{\times}$ and $r \in K$, then the corresponding distinct coordinates $x_i, x_j \in \mathcal{X}_\ell$ and $x_i', x_j' \in \mathcal{X}_\ell'$ satisfy $v_K(x_i - x_j) = 2v_K(u) + v_K(x'_i - x'_j)$. Consequently, the absolute depths shift by $2v_K(u)$, and the relative depths $\delta_{\mathfrak{s}}$ are invariant.
    \end{lem}
    
    \begin{proof}
    Suppose the Weierstrass models $E$ and $E'$ are related by an admissible change of variables $x' = u^2 x + r$ and $y' = u^3 y + u^2 s x + t$ over $K$. This induces a canonical bijection between their $\ell$-torsion $x$-coordinates. For any distinct $x_i, x_j \in \mathcal{X}_\ell$ mapping to $x'_i, x'_j \in \mathcal{X}'_\ell$, we have:
    $$
    v_K(x'_i - x'_j) = v_K(u^2(x_i - x_j)) = 2v_K(u) + v_K(x_i - x_j).
    $$
    Thus, every absolute depth shifts uniformly by $2v_K(u)$ and the constant cancels in the relative depths $\delta_{\mathfrak{s}} = d_{\mathfrak{s}} - d_{P(\mathfrak{s})}$.
    \end{proof}
    
    We can now prove Theorem \ref{thm:potential_red} regarding potential reduction.
    
    \begin{proof}[Proof of Theorem \ref{thm:potential_red}]
    Assume that $E$ is given by a minimal Weierstrass model over $K$. By the semistable reduction theorem, $E$ attains either good or multiplicative reduction over a finite extension $L/K$ with ramification index $e$. 
    
    Over this extension $L$, there exists a minimal Weierstrass model $E'$ that possesses this semistable reduction. By Proposition \ref{prop:base_cases}, the shape of the $\ell$-torsion cluster picture of $E'$ is either good or multiplicative. Our original model $E$ is related to $E'$ by an admissible change of variables over $L$, meaning there exists some $u \in L^{\times}$ such that $x = u^2 x' + r$. Applying Lemma \ref{lem:transformation}, the distances between any distinct $x_i, x_j \in \mathcal{X}_\ell$ on the original model $E$ over $L$ are given by:
    $$
    v_L(x_i - x_j) = 2v_L(u) + v_L(x'_i - x'_j).
    $$
    Normalising back to the base field $K$ using the identity $v_L(\cdot) = e \cdot v_K(\cdot)$, every absolute depth shifts uniformly by the constant $2v_L(u)/e$. Because adding a uniform constant to all pairwise distances preserves the shape of the cluster picture, the cluster picture of $E$ is identical to that of $E'$.
    \end{proof}

    \addtocontents{toc}{\protect\setcounter{tocdepth}{1}}

    \section{Absolute Depths on Minimal Models} \label{sec:absolute_depths}

    Before we sort our cluster pictures by reduction type to prove Theorem \ref{thm:kodara_intro}, we determine the absolute depths of the clusters on minimal models. 

    \begin{lem}\label{lem:potentially_good}
      Let $E/K$ be an elliptic curve given by a minimal Weierstrass model with discriminant $\Delta_{\min}$. If $E$ has potentially good reduction, the absolute depth of its single cluster is $d_{\mathcal{X}_{\ell}} = v_K(\Delta_{\min})/6$.
    \end{lem}

    \begin{proof}
      By the semistable reduction theorem, $E$ acquires good reduction over a finite extension $L/K$ with ramification index $e$. Let $E'$ be the minimal Weierstrass model over $L$ that possesses this good reduction, and let $u \in L^\times$ be such that $x=u^2x'+r$ over $L$.
      
      Because $E'$ has good reduction, its minimal discriminant $\Delta'$ is a unit, meaning $v_L(\Delta') = 0$. The two discriminants are related to each other by $\Delta_{\min} = u^{12}\Delta'$. Taking valuations over $L$, we obtain:$$v_L(\Delta_{\min}) = 12v_L(u) + v_L(\Delta') = 12v_L(u) \implies v_L(u) = \frac{v_L(\Delta_{\min})}{12}.$$
      
      By Proposition \ref{prop:base_cases}, the cluster picture of $E'$ has a single cluster of depth 0. Applying Lemma \ref{lem:transformation}, the distance between any two distinct $x$-coordinates on $E[\ell]$ is $v_L(x_i - x_j) = 2v_L(u) + 0$. Normalising this distance back to the base field $K$, we obtain 
      
      $$d_{\mathcal{X}_{\ell}} = \frac{1}{e} v_L(x_i - x_j) = \frac{2}{e} \left( \frac{e \cdot v_K(\Delta_{\min})}{12} \right) = \frac{v_K(\Delta_{\min})}{6}.$$
    \end{proof}

    \begin{lem}\label{lem:potentially_multiplicative}
      Let $E/K$ be an elliptic curve given by a minimal Weierstrass model. Assume $E$ has additive multiplicative reduction and the residue characteristic is $p \neq 2$. Then $d_{\mathcal{X}_{\ell}}=1$ and $\delta_{s_i}=n/\ell$ (where $n=-v_K(j(E))$) for every subcluster $\mathfrak{s}_i \subset \mathcal{X}_{\ell}$, and $i\in \{1, \ldots, \lfloor \ell/2 \rfloor \}$.
    \end{lem}

    \begin{proof}

      By the theory of Tate curves, an elliptic curve with reduction type $I_n^*$ is a quadratic twist of a curve with multiplicative reduction. It acquires split multiplicative reduction over a ramified quadratic extension $L/K$ (so $e=2$), and over $L$ its reduction type becomes $I_{2n}$. Let $E'$ be the minimal Weierstrass model over $L$ possessing this semistable reduction of type $I_{2n}$. The models $E$ and $E'$ are related over $L$ by an admissible change of variables, meaning their discriminants satisfy $\Delta_{min} = u^{12}\Delta'$ for some $u \in L^\times$. This means that $v_L(\Delta') = 2n$. Further, for $p \neq 2$, Tate's algorithm gives $v_K(\Delta_{\min}) = n + 6$, so, we obtain
      $$12v_L(u) = 2(n + 6) - 2n = 12 \implies v_L(u) = 1.$$
      Applying this to Lemma \ref{lem:transformation}, we find that absolute depths of the cluster picture of $E$ over $L$ shift from the absolute depths of the cluster picture of $E'$, by $v_L(u)$. Normalising back to $K$ and dividing by $e=2$, the lemma follows.
    \end{proof}

    \begin{rem}
      If $E$ has additive multiplicative reduction at $p = 2$, Tate's algorithm does not give a formula for the minimal discriminant, complicating the calculation of $v_L(u)$. A different approach is provided in Section \ref{subsec:pot_mult}.
    \end{rem}

    \section{Classification by Reduction Type}\label{sec:classification}

    The purpose of this section is to prove Theorem \ref{thm:kodara_intro}. We shall use the results from Section \ref{sec:absolute_depths} to classify Kodaira types from  $\ell$-torsion cluster pictures.

    \subsection{Potentially Good reduction}\label{subsec:pot_good}
    From Theorem \ref{thm:potential_red}, if the cluster picture consists of a single cluster, then $E$ has potentially good reduction. Let $d$ be its absolute depth, computed from an arbitrary Weierstrass model. By Lemma \ref{lem:potentially_good} and Lemma \ref{lem:transformation},
    $$d = \frac{v_K(\Delta_{\min})}{6} + 2k \quad \text{for some } k \in \mathbb{Z},$$
    which implies the congruence $v_K(\Delta_{\min}) \equiv 6d \pmod{12}$

\begin{enumerate}[label=(\roman*)]
\item \textit{Tame case ($p \ge 5$):} Tate's algorithm gives a one-to-one correspondence between $v_K(\Delta_{\min})$ and the Kodaira type. In particular,
$v_K(\Delta_{\min}) \in \{0,2,3,4,6,8,9,10\},$ so
the computed value of $6d \pmod{12}$ uniquely determines $v_K(\Delta_{\min})$.

\item \textit{Wild case ($p=2,3$):} Recall that, at $p=2$, we assume the curve does not have potentially good reduction of type $I_n^*$. For $p=2,3$, the correspondence between Kodaira types and $v_K(\Delta_{\min})$ is no longer injective, and so we use the conductor exponent $v_K(N)$. If $v_K(N)=0$, the curve has good reduction. Otherwise, by Ogg's formula, $v_K(\Delta_{\min}) - v_K(N) = m - 1$ (where $m$ is the number of irreducible components in the special fibre) is a constant, as is summarised below:

\begin{center}
  Table \ref{sec:classification}.1
  \vspace{0.1in}

  \renewcommand{\arraystretch}{1.2}
\begin{tabular}{|c|c|c|c|c|c|c|c|c|c|}
  
  \hline
  Kodaira type  & $I_n$  & $II$ & $III$ & $IV$ & $I_0^*$ & $I_n^*$ & $IV^*$ & $III^*$ & $II^*$ \\
  \hline
  $v_K(\Delta_{\min})-v_K(N)$ & $n-1$ & $0$ & $1$ & $2$ & $4$ & $n+4$ & $6$ & $7$ & $8$ \\
  \hline
\end{tabular}

\end{center}
\vspace{0.3cm}
from which, it follows that $0 \le v_K(\Delta_{\min}) - v_K(N) \le 8$. Combining this with $6d = v_K(\Delta_{\min}) + 12k$ gives
$$
v_K(N) \le 6d - 12k \le v_K(N) + 8.
$$
Since this interval has length $8 < 12$, there is a unique $k \in \Z$ satisfying the inequality. Adjusting by this $k$ recovers $v_K(\Delta_{\min})$, and the Kodaira type then follows from $v_K(\Delta_{\min}) - v_K(N)$.
\end{enumerate}

\subsection{Distinguishing multiplicative reduction and additive multiplicative reduction}\label{subsec:differentiating}
By Theorem \ref{thm:potential_red}, if the $\ell$-torsion cluster picture is multiplicative, the curve possesses either multiplicative reduction (Type $I_n$) or additive multiplicative reduction (Type $I_n^*$). Because the shape of the cluster picture of both reduction types is identical, we outline a way to distinguish them based on residue characteristic.
    
    \begin{enumerate}[label=(\alph*)]
        \item \textit{Tame case ($p \neq 2$):} The parity of the absolute depth of the outermost cluster $d_{\mathcal{X}_{\ell}}$ distinguishes multiplicative from additive multiplicative reduction.
        
        \begin{lem} Let $p \neq 2$.  If $d_{\mathcal{X}_\ell} \equiv 0 \pmod 2$, the curve has multiplicative reduction (Type $I_n$).  If $d_{\mathcal{X}_\ell} \equiv 1 \pmod 2$, the curve has additive multiplicative reduction (Type $I_n^*$).
          \end{lem}
          
          \begin{proof}
            If $E$ is given by a minimal Weierstrass equation, then the outermost depth $d_{\mathcal{X}_{\ell}}$ is $0$ in the multiplicative case by Proposition \ref{prop:base_cases}, and $1$ in the additive case by Lemma \ref{lem:potentially_multiplicative}. By Lemma \ref{lem:transformation}, an arbitrary change of coordinates to a non-minimal model shifts all the absolute depths by an even integer $2k$, leaving their parity invariant.
          \end{proof}

      \item \textit{Wild case ($p=2$):} We distinguish multiplicative from additive multiplicative reduction using the action of inertia on the $\ell$-torsion.
      
      Let $\bar{\rho}: \text{Gal}(\overline{K}/K) \to \text{GL}_2(\mathbb{F}_\ell)$ be the mod-$\ell$ Galois representation attached to $E[\ell]$.

      \begin{enumerate}[label=(\roman*)]
        \item \textit{Multiplicative Reduction (Type $I_n$):} $E$ has multiplicative reduction over $K$ if and only if $\bar{\rho}(I_K)$  consists of unipotent matrices; equivalently, every $\sigma \in I_K$ has characteristic polynomial $(x-1)^2 \pmod \ell$.
        
        \item \textit{Additive Multiplicative Reduction (Type $I_n^*$):} $E$ has potentially multiplicative reduction over $K$ if and only if $\bar{\rho}(I_K)$ contains elements acting with the characteristic polynomial $(x+1)^2 \pmod \ell$. 
    \end{enumerate}
    
    This distinction arises from the theory of Tate curves. A curve of type $I_n^*$ is a quadratic twist of a split multiplicative curve by a ramified quadratic character, say $\psi$. Since the extension is ramified, there exists $\sigma \in I_K$ with  $\psi(\sigma)=-1$, so the corresponding inertia action has eigenvalues $-1$. In particular, $\bar{\rho}(\sigma)$ is not completely unipotent, which distinguishes additive multiplicative reduction from multiplicative reduction.

    \end{enumerate}

    \subsection{Multiplicative reduction}\label{subsec:mult}
    The relative depths of the nested subclusters are $n/\ell$. Because the relative depths $\delta_{\mathfrak{s}}$ are invariant under model changes, scaling $\delta_{\mathfrak{s}}$ by $\ell$ yields the integer $n$.

    \subsection{Additive multiplicative reduction}\label{subsec:pot_mult}
     Recovering $n$ depends on the ramification of the quadratic extension $L/K$ over which $E$ acquires split multiplicative reduction.

    \begin{enumerate}[label=(\roman*)]
      \item \textit{Tame Ramification ($p \ne 2$):} Just as in the multiplicative case, the relative depths of the inner clusters are $n/\ell$, so scaling by $\ell$ recovers $n$.
      
      \item \textit{Wild Ramification ($p = 2$):}  A curve of type $I_n^*$ is a quadratic twist of a multiplicative curve, acquiring split multiplicative reduction over a ramified quadratic extension $L/K$ \cite[Lemma 5.2, V]{SilII}. Lorenzini \cite[Theorem 2.8]{Dinomult} shows that over $L$, the curve attains reduction type $I_{2\nu}$ (where $\nu = -v_K(j(E))$), and proves that $n = \nu + 4(s_{L/K}-1)$. Here, $s_{L/K}$ is the valuation of the different of the extension $\mathcal{O}_L/\mathcal{O}_K$, defined via the higher ramification groups of $\text{Gal}(L/K)$ by $s_{L/K} = \sum_{i=0}^{\infty} (|H_i| - 1)$ \cite[IV.2, Proposition 4]{LocalFields}. Therefore, to determine $n$, we pass to the semistable extension $L$, where the relative depths of the cluster picture over $L$ yields $\nu$, which, combined with $s_{L/K}$, recovers the integer $n$ for the original curve over $K$.
    \end{enumerate}

    \begin{proof}[Proof of Theorem \ref{thm:kodara_intro}]
      The classification of potentially good reduction follows from the results in Section \ref{subsec:pot_good}. For curves whose $\ell$-torsion cluster picture is not good, the curve is potentially multiplicative from Theorem \ref{thm:potential_red}. The multiplicative and additive multiplicative cases are distinguished by the parity of $d_{\mathcal{X}_{\ell}}$ if the residue characteristic is $p \ne 2$, or by the action of inertia if $p=2$, as shown in Section \ref{subsec:differentiating}. Finally, the Kodaira types $I_n$ and $I_n^*$ are recovered by scaling the relative depths as detailed in Sections \ref{subsec:mult} and \ref{subsec:pot_mult}. This exhausts all cases, completing the proof.
    \end{proof}

    \printbibliography

    \end{document}